\documentclass[12pt]{amsart}
\usepackage{amsmath,amssymb,amsthm}
\usepackage{mathtools}
\usepackage{amssymb}

\usepackage{stmaryrd} 
\usepackage[inline]{enumitem}
\usepackage{graphicx,setspace}
\usepackage{tikz}
\usetikzlibrary{arrows}

\usepackage{wasysym}
\usepackage[makeroom]{cancel}
\usepackage{fullpage}
\usepackage{comment}
\usepackage{enumitem}
\usepackage{caption}
\usepackage{mathtools}
\usepackage{url}
\usepackage{tabto}

\allowdisplaybreaks{}

\theoremstyle{plain}
\newtheorem{theorem}{Theorem}[section]

\newtheorem{lemma}[theorem]{Lemma}

\newtheorem{corollary}[theorem]{Corollary}

\newtheorem*{claim*}{Claim}
\theoremstyle{definition}
\newtheorem{remark}[theorem]{Remark}
\newtheorem{definition}[theorem]{Definition}
\newtheorem{example}[theorem]{Example}
 
\usepackage[utf8]{inputenc}

\DeclareUnicodeCharacter{03B1}{\ensuremath{\alpha}}
\DeclareUnicodeCharacter{03B9}{\ensuremath{\iota}}
\DeclareUnicodeCharacter{03B2}{\ensuremath{\beta}}
\DeclareUnicodeCharacter{03BA}{\ensuremath{\kappa}}
\DeclareUnicodeCharacter{03F0}{\ensuremath{\varkappa}}
\DeclareUnicodeCharacter{03C3}{\ensuremath{\sigma}}
\DeclareUnicodeCharacter{03B3}{\ensuremath{\gamma}}
\DeclareUnicodeCharacter{03BB}{\ensuremath{\lambda}}
\DeclareUnicodeCharacter{03B4}{\ensuremath{\delta}}
\DeclareUnicodeCharacter{03BC}{\ensuremath{\mu}} 
\DeclareUnicodeCharacter{00B5}{\ensuremath{\mu}} 
\DeclareUnicodeCharacter{03C4}{\ensuremath{\tau}}
\DeclareUnicodeCharacter{03BD}{\ensuremath{\nu}}
\DeclareUnicodeCharacter{03C5}{\ensuremath{\upsilon}}
\DeclareUnicodeCharacter{03F5}{\ensuremath{\epsilon}}
\DeclareUnicodeCharacter{03B5}{\ensuremath{\varepsilon}}
\DeclareUnicodeCharacter{03BE}{\ensuremath{\xi}}
\DeclareUnicodeCharacter{03B6}{\ensuremath{\zeta}}
\DeclareUnicodeCharacter{03D5}{\ensuremath{\phi}}
\DeclareUnicodeCharacter{03C6}{\ensuremath{\varphi}}
\DeclareUnicodeCharacter{03B7}{\ensuremath{\eta}}
\DeclareUnicodeCharacter{03C0}{\ensuremath{\pi}}
\DeclareUnicodeCharacter{03D6}{\ensuremath{\varpi}}
\DeclareUnicodeCharacter{03C7}{\ensuremath{\chi}}
\DeclareUnicodeCharacter{03B8}{\ensuremath{\theta}}
\DeclareUnicodeCharacter{03C8}{\ensuremath{\psi}}
\DeclareUnicodeCharacter{03D1}{\ensuremath{\vartheta}}
\DeclareUnicodeCharacter{03C1}{\ensuremath{\rho}}
\DeclareUnicodeCharacter{03F1}{\ensuremath{\varrho}}
\DeclareUnicodeCharacter{03C9}{\ensuremath{\omega}}
\DeclareUnicodeCharacter{0393}{\ensuremath{\Gamma}}
\DeclareUnicodeCharacter{039E}{\ensuremath{\Xi}}
\DeclareUnicodeCharacter{03A6}{\ensuremath{\Phi}}
\DeclareUnicodeCharacter{0394}{\ensuremath{\Delta}}
\DeclareUnicodeCharacter{03A0}{\ensuremath{\Pi}}
\DeclareUnicodeCharacter{03A8}{\ensuremath{\Psi}}
\DeclareUnicodeCharacter{0398}{\ensuremath{\Theta}}
\DeclareUnicodeCharacter{03A3}{\ensuremath{\Sigma}}
\DeclareUnicodeCharacter{03A9}{\ensuremath{\Omega}}
\DeclareUnicodeCharacter{039B}{\ensuremath{\Lambda}}

\DeclareUnicodeCharacter{2207}{\ensuremath{\nabla}}
\DeclareUnicodeCharacter{2202}{\ensuremath{\partial}}
\DeclareUnicodeCharacter{222C}{\ensuremath{\iint}}
\DeclareUnicodeCharacter{222D}{\ensuremath{\iiint}}
\DeclareUnicodeCharacter{2A0C}{\ensuremath{\iiiint}}
\DeclareUnicodeCharacter{222E}{\ensuremath{\oint}}
\DeclareUnicodeCharacter{222F}{\ensuremath{\oiint}}
\DeclareUnicodeCharacter{2230}{\ensuremath{\oiiint}}

\DeclareUnicodeCharacter{2254}{:=}
\DeclareUnicodeCharacter{2255}{=:}
\DeclareUnicodeCharacter{2203}{\ensuremath{\exists}}
\DeclareUnicodeCharacter{2200}{\ensuremath{\forall}}

\DeclareUnicodeCharacter{22A4}{\top}                    
\DeclareUnicodeCharacter{22A5}{\bot}                    
\DeclareUnicodeCharacter{2135}{\ensuremath{\aleph}}
\DeclareUnicodeCharacter{2211}{\ensuremath{\sum}}
\DeclareUnicodeCharacter{222B}{\ensuremath{\int}}
\DeclareUnicodeCharacter{220F}{\ensuremath{\prod}}
\DeclareUnicodeCharacter{2261}{\ensuremath{\equiv}}
\DeclareUnicodeCharacter{2020}{\ensuremath{\dagger}} 
\DeclareUnicodeCharacter{00AC}{\ensuremath{\neg}}
\DeclareUnicodeCharacter{33D1}{\ensuremath{\ln}}
\DeclareUnicodeCharacter{33D2}{\ensuremath{\log}}
\DeclareUnicodeCharacter{221D}{\propto} 
\DeclareUnicodeCharacter{211C}{\ensuremath{\Re}}
\DeclareUnicodeCharacter{2111}{\ensuremath{\Im}}
\DeclareUnicodeCharacter{00A0}{~}
\DeclareUnicodeCharacter{202F}{\,}
\DeclareUnicodeCharacter{230A}{\ensuremath{\lfloor}}
\DeclareUnicodeCharacter{230B}{\ensuremath{\rfloor}}
\DeclareUnicodeCharacter{221A}{\ensuremath{\sqrt}}
\DeclareUnicodeCharacter{221B}{\ensuremath{\sqrt[3]}}
\DeclareUnicodeCharacter{221C}{\ensuremath{\sqrt[4]}}
\DeclareUnicodeCharacter{00D7}{\ensuremath{\times}}
\DeclareUnicodeCharacter{00F7}{\ensuremath{\div}}
\DeclareUnicodeCharacter{00B1}{\ensuremath{\pm}}
\DeclareUnicodeCharacter{2213}{\ensuremath{\mp}}
\DeclareUnicodeCharacter{2215}{\ensuremath{/}}
\DeclareUnicodeCharacter{2212}{\ensuremath{-}}
\DeclareUnicodeCharacter{221E}{\ensuremath{\infty}}
\DeclareUnicodeCharacter{2228}{\ensuremath{\vee}}
\DeclareUnicodeCharacter{2227}{\ensuremath{\wedge}}

\DeclareUnicodeCharacter{22C1}{\ensuremath{\bigvee}}
\DeclareUnicodeCharacter{22C0}{\ensuremath{\bigwedge}}
\DeclareUnicodeCharacter{22C3}{\ensuremath{\bigcup}}
\DeclareUnicodeCharacter{22C2}{\ensuremath{\bigcap}}
\DeclareUnicodeCharacter{2A00}{\ensuremath{\bigodot}}
\DeclareUnicodeCharacter{2A01}{\ensuremath{\bigoplus}}
\DeclareUnicodeCharacter{2A02}{\ensuremath{\bigotimes}}

\DeclareUnicodeCharacter{2260}{\ensuremath{\neq}}
\DeclareUnicodeCharacter{2248}{\ensuremath{\approx}}
\DeclareUnicodeCharacter{2264}{\ensuremath{\leq}}
\DeclareUnicodeCharacter{2265}{\ensuremath{\geq}}
\DeclareUnicodeCharacter{226A}{\ensuremath{\ll}}
\DeclareUnicodeCharacter{226B}{\ensuremath{\gg}}
\DeclareUnicodeCharacter{226D}{\not\equiv}              
\DeclareUnicodeCharacter{2270}{\not\le}                 
\DeclareUnicodeCharacter{2271}{\not\ge}                 
\DeclareUnicodeCharacter{2272}{\lesssim}                
\DeclareUnicodeCharacter{2273}{\gtrsim}                 
\DeclareUnicodeCharacter{227A}{\prec}                   
\DeclareUnicodeCharacter{2282}{\subset}                 
\DeclareUnicodeCharacter{2286}{\subseteq}               

\DeclareUnicodeCharacter{220B}{\ensuremath{\ni}}
\DeclareUnicodeCharacter{2205}{\ensuremath{\emptyset}}
\DeclareUnicodeCharacter{2208}{\ensuremath{\in}}
\DeclareUnicodeCharacter{2282}{\ensuremath{\subset}}
\DeclareUnicodeCharacter{2283}{\ensuremath{\supset}}
\DeclareUnicodeCharacter{2286}{\ensuremath{\subseteq}}
\DeclareUnicodeCharacter{2287}{\ensuremath{\supseteq}}
\DeclareUnicodeCharacter{2229}{\ensuremath{\cap}}
\DeclareUnicodeCharacter{222A}{\ensuremath{\cup}}
\DeclareUnicodeCharacter{2288}{\ensuremath{\nsubseteq}} 

\DeclareUnicodeCharacter{2026}{\ifmmode\ldots\else\textellipsis\fi} 
\DeclareUnicodeCharacter{22C5}{\ensuremath{\cdot}}
\DeclareUnicodeCharacter{22EE}{\ensuremath{\vdots}}                  

\DeclareUnicodeCharacter{21D2}{\ensuremath{\Rightarrow}}
\DeclareUnicodeCharacter{21D0}{\ensuremath{\Leftarrow}}
\DeclareUnicodeCharacter{21D4}{\ensuremath{\Leftrightarrow}}
\DeclareUnicodeCharacter{2192}{\ensuremath{\to}}
\DeclareUnicodeCharacter{2190}{\ensuremath{\gets}}
\DeclareUnicodeCharacter{21A6}{\ensuremath{\mapsto}}

\DeclareUnicodeCharacter{20AC}{\EUR}

\DeclareUnicodeCharacter{2070}{\ensuremath{^0}}
\DeclareUnicodeCharacter{00B9}{\ensuremath{^1}}
\DeclareUnicodeCharacter{00B2}{\ensuremath{^2}}
\DeclareUnicodeCharacter{00B3}{\ensuremath{^3}}
\DeclareUnicodeCharacter{2074}{\ensuremath{^4}}
\DeclareUnicodeCharacter{2075}{\ensuremath{^5}}
\DeclareUnicodeCharacter{2076}{\ensuremath{^6}}
\DeclareUnicodeCharacter{2077}{\ensuremath{^7}}
\DeclareUnicodeCharacter{2078}{\ensuremath{^8}}
\DeclareUnicodeCharacter{2079}{\ensuremath{^9}}
\DeclareUnicodeCharacter{207A}{\ensuremath{^+}}
\DeclareUnicodeCharacter{207B}{\ensuremath{^-}}
\DeclareUnicodeCharacter{207C}{\ensuremath{^=}}
\DeclareUnicodeCharacter{207D}{\ensuremath{^(}}
\DeclareUnicodeCharacter{207E}{\ensuremath{^)}}
\DeclareUnicodeCharacter{2080}{\ensuremath{_0}}
\DeclareUnicodeCharacter{2081}{\ensuremath{_1}}
\DeclareUnicodeCharacter{2082}{\ensuremath{_2}}
\DeclareUnicodeCharacter{2083}{\ensuremath{_3}}
\DeclareUnicodeCharacter{2084}{\ensuremath{_4}}
\DeclareUnicodeCharacter{2085}{\ensuremath{_5}}
\DeclareUnicodeCharacter{2086}{\ensuremath{_6}}
\DeclareUnicodeCharacter{2087}{\ensuremath{_7}}
\DeclareUnicodeCharacter{2088}{\ensuremath{_8}}
\DeclareUnicodeCharacter{2089}{\ensuremath{_9}}
\DeclareUnicodeCharacter{208A}{\ensuremath{_+}}
\DeclareUnicodeCharacter{208B}{\ensuremath{_-}}
\DeclareUnicodeCharacter{208C}{\ensuremath{_=}}
\DeclareUnicodeCharacter{208D}{\ensuremath{_(}}
\DeclareUnicodeCharacter{208E}{\ensuremath{_)}}

\DeclareUnicodeCharacter{2096}{_k}
\DeclareUnicodeCharacter{2097}{_l}
\DeclareUnicodeCharacter{2098}{_m}
\DeclareUnicodeCharacter{2099}{_n}
\DeclareUnicodeCharacter{2099}{_n}
\DeclareUnicodeCharacter{209B}{_s}
\DeclareUnicodeCharacter{209C}{_t}
\DeclareUnicodeCharacter{1D64}{_u}

\DeclareUnicodeCharacter{1D538}{\ensuremath{\mathbb{A}}}
\DeclareUnicodeCharacter{1D539}{\ensuremath{\mathbb{B}}}
\DeclareUnicodeCharacter{02102}{\ensuremath{\mathbb{C}}}
\DeclareUnicodeCharacter{1D53B}{\ensuremath{\mathbb{D}}}
\DeclareUnicodeCharacter{1D53C}{\ensuremath{\mathbb{E}}}
\DeclareUnicodeCharacter{1D53D}{\ensuremath{\mathbb{F}}}
\DeclareUnicodeCharacter{1D53E}{\ensuremath{\mathbb{G}}}
\DeclareUnicodeCharacter{0210D}{\ensuremath{\mathbb{H}}}
\DeclareUnicodeCharacter{1D540}{\ensuremath{\mathbb{I}}}
\DeclareUnicodeCharacter{1D541}{\ensuremath{\mathbb{J}}}
\DeclareUnicodeCharacter{1D542}{\ensuremath{\mathbb{K}}}
\DeclareUnicodeCharacter{1D543}{\ensuremath{\mathbb{L}}}
\DeclareUnicodeCharacter{1D544}{\ensuremath{\mathbb{M}}}
\DeclareUnicodeCharacter{02115}{\ensuremath{\mathbb{N}}}
\DeclareUnicodeCharacter{1D546}{\ensuremath{\mathbb{O}}}
\DeclareUnicodeCharacter{02119}{\ensuremath{\mathbb{P}}}
\DeclareUnicodeCharacter{0211A}{\ensuremath{\mathbb{Q}}}
\DeclareUnicodeCharacter{0211D}{\ensuremath{\mathbb{R}}}
\DeclareUnicodeCharacter{1D54A}{\ensuremath{\mathbb{S}}}
\DeclareUnicodeCharacter{1D54B}{\ensuremath{\mathbb{T}}}
\DeclareUnicodeCharacter{1D54C}{\ensuremath{\mathbb{U}}}
\DeclareUnicodeCharacter{1D54D}{\ensuremath{\mathbb{V}}}
\DeclareUnicodeCharacter{1D54E}{\ensuremath{\mathbb{W}}}
\DeclareUnicodeCharacter{1D54F}{\ensuremath{\mathbb{X}}}
\DeclareUnicodeCharacter{1D550}{\ensuremath{\mathbb{Y}}}
\DeclareUnicodeCharacter{02124}{\ensuremath{\mathbb{Z}}}
\DeclareUnicodeCharacter{1D552}{\ensuremath{\mathbb{a}}}
\DeclareUnicodeCharacter{1D553}{\ensuremath{\mathbb{b}}}
\DeclareUnicodeCharacter{1D554}{\ensuremath{\mathbb{c}}}
\DeclareUnicodeCharacter{1D555}{\ensuremath{\mathbb{d}}}
\DeclareUnicodeCharacter{1D556}{\ensuremath{\mathbb{e}}}
\DeclareUnicodeCharacter{1D557}{\ensuremath{\mathbb{f}}}
\DeclareUnicodeCharacter{1D558}{\ensuremath{\mathbb{g}}}
\DeclareUnicodeCharacter{1D559}{\ensuremath{\mathbb{h}}}
\DeclareUnicodeCharacter{1D55A}{\ensuremath{\mathbb{i}}}
\DeclareUnicodeCharacter{1D55B}{\ensuremath{\mathbb{j}}}
\DeclareUnicodeCharacter{1D55C}{\ensuremath{\mathbb{k}}}
\DeclareUnicodeCharacter{1D55D}{\ensuremath{\mathbb{l}}}
\DeclareUnicodeCharacter{1D55E}{\ensuremath{\mathbb{m}}}
\DeclareUnicodeCharacter{1D55F}{\ensuremath{\mathbb{n}}}
\DeclareUnicodeCharacter{1D560}{\ensuremath{\mathbb{o}}}
\DeclareUnicodeCharacter{1D561}{\ensuremath{\mathbb{p}}}
\DeclareUnicodeCharacter{1D562}{\ensuremath{\mathbb{q}}}
\DeclareUnicodeCharacter{1D563}{\ensuremath{\mathbb{r}}}
\DeclareUnicodeCharacter{1D564}{\ensuremath{\mathbb{s}}}
\DeclareUnicodeCharacter{1D565}{\ensuremath{\mathbb{t}}}
\DeclareUnicodeCharacter{1D566}{\ensuremath{\mathbb{u}}}
\DeclareUnicodeCharacter{1D567}{\ensuremath{\mathbb{v}}}
\DeclareUnicodeCharacter{1D568}{\ensuremath{\mathbb{w}}}
\DeclareUnicodeCharacter{1D569}{\ensuremath{\mathbb{x}}}
\DeclareUnicodeCharacter{1D56A}{\ensuremath{\mathbb{y}}}
\DeclareUnicodeCharacter{1D56B}{\ensuremath{\mathbb{z}}}
\DeclareUnicodeCharacter{1D7D8}{\ensuremath{\mathbb{0}}}
\DeclareUnicodeCharacter{1D7D9}{\ensuremath{\mathbb{1}}}
\DeclareUnicodeCharacter{1D7DA}{\ensuremath{\mathbb{2}}}
\DeclareUnicodeCharacter{1D7DB}{\ensuremath{\mathbb{3}}}
\DeclareUnicodeCharacter{1D7DC}{\ensuremath{\mathbb{4}}}
\DeclareUnicodeCharacter{1D7DD}{\ensuremath{\mathbb{5}}}
\DeclareUnicodeCharacter{1D7DE}{\ensuremath{\mathbb{6}}}
\DeclareUnicodeCharacter{1D7DF}{\ensuremath{\mathbb{7}}}
\DeclareUnicodeCharacter{1D7E0}{\ensuremath{\mathbb{8}}}
\DeclareUnicodeCharacter{1D7E1}{\ensuremath{\mathbb{9}}}

\DeclareUnicodeCharacter{1D504}{\ensuremath{\mathfrak{A}}}
\DeclareUnicodeCharacter{1D505}{\ensuremath{\mathfrak{B}}}
\DeclareUnicodeCharacter{1D507}{\ensuremath{\mathfrak{D}}}
\DeclareUnicodeCharacter{1D508}{\ensuremath{\mathfrak{E}}}
\DeclareUnicodeCharacter{1D509}{\ensuremath{\mathfrak{F}}}
\DeclareUnicodeCharacter{1D50A}{\ensuremath{\mathfrak{G}}}
\DeclareUnicodeCharacter{1D50D}{\ensuremath{\mathfrak{J}}}
\DeclareUnicodeCharacter{1D50E}{\ensuremath{\mathfrak{K}}}
\DeclareUnicodeCharacter{1D50F}{\ensuremath{\mathfrak{L}}}
\DeclareUnicodeCharacter{1D510}{\ensuremath{\mathfrak{M}}}
\DeclareUnicodeCharacter{1D511}{\ensuremath{\mathfrak{N}}}
\DeclareUnicodeCharacter{1D512}{\ensuremath{\mathfrak{O}}}
\DeclareUnicodeCharacter{1D513}{\ensuremath{\mathfrak{P}}}
\DeclareUnicodeCharacter{1D514}{\ensuremath{\mathfrak{Q}}}
\DeclareUnicodeCharacter{1D516}{\ensuremath{\mathfrak{S}}}
\DeclareUnicodeCharacter{1D517}{\ensuremath{\mathfrak{T}}}
\DeclareUnicodeCharacter{1D518}{\ensuremath{\mathfrak{U}}}
\DeclareUnicodeCharacter{1D519}{\ensuremath{\mathfrak{V}}}
\DeclareUnicodeCharacter{1D51A}{\ensuremath{\mathfrak{W}}}
\DeclareUnicodeCharacter{1D51B}{\ensuremath{\mathfrak{X}}}
\DeclareUnicodeCharacter{1D51C}{\ensuremath{\mathfrak{Y}}}

\DeclareUnicodeCharacter{1D4D0}{\ensuremath{\mathcal{A}}}
\DeclareUnicodeCharacter{1D4D1}{\ensuremath{\mathcal{B}}}
\DeclareUnicodeCharacter{1D4D2}{\ensuremath{\mathcal{C}}}
\DeclareUnicodeCharacter{1D4D3}{\ensuremath{\mathcal{D}}}
\DeclareUnicodeCharacter{1D4D4}{\ensuremath{\mathcal{E}}}
\DeclareUnicodeCharacter{1D4D5}{\ensuremath{\mathcal{F}}}
\DeclareUnicodeCharacter{1D4D6}{\ensuremath{\mathcal{G}}}
\DeclareUnicodeCharacter{1D4D7}{\ensuremath{\mathcal{H}}}
\DeclareUnicodeCharacter{1D4D8}{\ensuremath{\mathcal{I}}}
\DeclareUnicodeCharacter{1D4D9}{\ensuremath{\mathcal{J}}}
\DeclareUnicodeCharacter{1D4DA}{\ensuremath{\mathcal{K}}}
\DeclareUnicodeCharacter{1D4DB}{\ensuremath{\mathcal{L}}}
\DeclareUnicodeCharacter{1D4DC}{\ensuremath{\mathcal{M}}}
\DeclareUnicodeCharacter{1D4DD}{\ensuremath{\mathcal{N}}}
\DeclareUnicodeCharacter{1D4DE}{\ensuremath{\mathcal{O}}}
\DeclareUnicodeCharacter{1D4DF}{\ensuremath{\mathcal{P}}}
\DeclareUnicodeCharacter{1D4E0}{\ensuremath{\mathcal{Q}}}
\DeclareUnicodeCharacter{1D4E1}{\ensuremath{\mathcal{R}}}
\DeclareUnicodeCharacter{1D4E2}{\ensuremath{\mathcal{S}}}
\DeclareUnicodeCharacter{1D4E3}{\ensuremath{\mathcal{T}}}
\DeclareUnicodeCharacter{1D4E4}{\ensuremath{\mathcal{U}}}
\DeclareUnicodeCharacter{1D4E5}{\ensuremath{\mathcal{V}}}
\DeclareUnicodeCharacter{1D4E6}{\ensuremath{\mathcal{W}}}
\DeclareUnicodeCharacter{1D4E7}{\ensuremath{\mathcal{X}}}
\DeclareUnicodeCharacter{1D4E8}{\ensuremath{\mathcal{Y}}}
\DeclareUnicodeCharacter{1D4E9}{\ensuremath{\mathcal{Z}}}

\DeclareUnicodeCharacter{1D43}{^a}
\DeclareUnicodeCharacter{1D47}{^b}
\DeclareUnicodeCharacter{1D9C}{^c}
\DeclareUnicodeCharacter{1D48}{^d}
\DeclareUnicodeCharacter{1D49}{^e}
\DeclareUnicodeCharacter{1DA0}{^f}
\DeclareUnicodeCharacter{1D4D}{^g}
\DeclareUnicodeCharacter{02B0}{^h}
\DeclareUnicodeCharacter{2071}{^i}
\DeclareUnicodeCharacter{02B2}{^j}
\DeclareUnicodeCharacter{1D4F}{^k}
\DeclareUnicodeCharacter{02E1}{^l}
\DeclareUnicodeCharacter{1D50}{^m}
\DeclareUnicodeCharacter{207F}{^n}
\DeclareUnicodeCharacter{1D52}{^o}
\DeclareUnicodeCharacter{1D56}{^p}
\DeclareUnicodeCharacter{02B3}{^r}
\DeclareUnicodeCharacter{02E2}{^s}
\DeclareUnicodeCharacter{1D57}{^t}
\DeclareUnicodeCharacter{1D58}{^u}
\DeclareUnicodeCharacter{1D5B}{^v}
\DeclareUnicodeCharacter{02B7}{^w}
\DeclareUnicodeCharacter{02E3}{^x}
\DeclareUnicodeCharacter{02B8}{^y}
\DeclareUnicodeCharacter{1DBB}{^z}
\DeclareUnicodeCharacter{1D2C}{^A}
\DeclareUnicodeCharacter{1D2E}{^B}
\DeclareUnicodeCharacter{1D30}{^D}
\DeclareUnicodeCharacter{1D31}{^E}
\DeclareUnicodeCharacter{1D33}{^G}
\DeclareUnicodeCharacter{1D34}{^H}
\DeclareUnicodeCharacter{1D35}{^I}
\DeclareUnicodeCharacter{1D36}{^J}
\DeclareUnicodeCharacter{1D37}{^K}
\DeclareUnicodeCharacter{1D38}{^L}
\DeclareUnicodeCharacter{1D39}{^M}
\DeclareUnicodeCharacter{1D3A}{^N}
\DeclareUnicodeCharacter{1D3C}{^O}
\DeclareUnicodeCharacter{1D3E}{^P}
\DeclareUnicodeCharacter{1D3F}{^R}
\DeclareUnicodeCharacter{1D40}{^T}
\DeclareUnicodeCharacter{1D41}{^U}
\DeclareUnicodeCharacter{1D42}{^W}

\DeclareUnicodeCharacter{1D49E}{\mathcal{C}}            
\DeclareUnicodeCharacter{1D49F}{\mathcal{D}}            
\DeclareUnicodeCharacter{1D4BB}{\mathcal{f}}            
\DeclareUnicodeCharacter{1D4DB}{\mathcal{L}}            
\DeclareUnicodeCharacter{1D4E3}{\mathcal{T}}            

\DeclareUnicodeCharacter{2308}{\left\lceil}
\DeclareUnicodeCharacter{2309}{\right\rceil}
\DeclareUnicodeCharacter{1D62}{\ensuremath{_i}}
\DeclareUnicodeCharacter{00B7}{\ensuremath{\cdot}}
\DeclareUnicodeCharacter{223C}{\ensuremath{\sim}}
\DeclareUnicodeCharacter{00D5}{\ifmmode\tilde{O}\else\~{O}\fi}
\DeclareUnicodeCharacter{2245}{\ensuremath{\cong}}
\DeclareUnicodeCharacter{22EF}{\ensuremath{\cdots}}
\DeclareUnicodeCharacter{2C7C}{\ensuremath{_j}}
\DeclareUnicodeCharacter{209A}{\ensuremath{_p}}
\DeclareUnicodeCharacter{2297}{\ensuremath{\otimes}}
\DeclareUnicodeCharacter{21AA}{\xhookrightarrow{}}
\DeclareUnicodeCharacter{2016}{\Vert}
\DeclareUnicodeCharacter{2113}{\ell}

\DeclareUnicodeCharacter{2223}{\mid}
\DeclareUnicodeCharacter{2224}{\nmid}

\DeclareUnicodeCharacter{220B}{\ensuremath{\ni}}
\DeclareUnicodeCharacter{2209}{\not\in}

\newcommand{\parens}[1]{\left(#1\right)}

\newcommand{\sqbracks}[1]{\left[#1\right]}

\newcommand{\norm}[1]{\left\lVert#1\right\rVert}

\DeclareMathOperator{\rad}{rad}
\DeclareMathOperator{\Disc}{Disc}

\DeclareMathOperator{\SD}{SD}

\newif\ifshowold{} 
\showoldfalse{}

\begin{document}

\title{The Ring Learning With Errors Problem: Spectral Distortion}

\author [L. Babinkostova] {L.\ Babinkostova $^1$}
\address{$^1$ Boise State University}

\author[A. Chin]{A.\ Chin $^2$}
\address{$^2$ University of California, Berkeley}

\author[A. Kirtland ]{A.\ Kirtland $^3$ }
\address{ $^3$ Washington University in St. Louis}

\author[V. Nazarchuk]{V. \ Nazarchuk $^4$}
\address{$^4$ Yale University}

\author[E. Plotnick]{E. \ Plotnick $^5$}
\address{$^5$ Harvard University}

\thanks{Supported by the National Science Foundation under the grant number DMS-1659872.}

\thanks{$^{\S}$ Corresponding Author: liljanababinkostova@boisestate.edu}

\maketitle

\begin{abstract}

We answer a question posed by Y. Elias and others~\cite{RLWE for NT} about possible spectral distortions of algebraic numbers.  We provide a closed form for the spectral distortion of certain classes of cyclotomic polynomials. Moreover, we present a bound on the spectral distortion of cyclotomic polynomials.

\end{abstract}

\section{Introduction}



A large fraction of lattice-based cryptographic constructions are built upon on Learning With Errors (LWE) problem  or its variants learning with errors. The Learning With Errors (LWE) problem introduced by O. Regev \cite{LP}, relates to solving a ``noisy" linear system modulo a known integer. The ``algebraically structured" variants, called RLWE \cite{MR}, PLWE \cite{LP}, Module-LWE \cite {A2}. As other cryptographic problems, LWE is an average-case problem which means the input instances are chosen at random from a prescribed probability distribution. 

Since its introduction, the RLWE problem \cite{LP} has already been used as a building block for many cryptographic applications. It has since been used as a hardness assumption in the constructions of efficient signature schemes \cite{W}, fully-homomorphic encryption schemes \cite{BV}, pseudo-random functions \cite{B}, protocols for secure multi-party computation \cite{DPA}, and also gives an explanation for the hardness of the NTRU cryptosystem \cite{HPS}. 

The RLWE and PLWE problems are formulated as either ``search" or
``decision" problems. Let $f(x) \in \mathbb{Z}[x]$ to be monic and irreducible of degree $n$, $P=\mathbb{Z}[x]/f(x)$, and $P_q = P/qP \cong F_q[x]/f(x)$ where $q$ is a prime.

{\bf Search PLWE Problem.} Let $s(x)\in P_q$ be a secret. The
search PLWE problem, is to discover $s(x)$ given access to arbitrarily many independent samples of the form $(a_i(x), b_i(x) = a_i(x)s(x)+e_i(x)) \in  P_q\times P_q$,
where for each $i$, $e_i(x)$ is chosen from a discretized Gaussian of parameter $\sigma$, and $a_i(x)$ is uniformly random.
The polynomial $s(x)$ is the secret and the polynomials $e_i(x)$ are the errors.

{\bf Decision PLWE Problem.} Let $s(x) \in P_q$ be a secret. The
decision PLWE problem is to distinguish, with non-negligible advantage, between the same number of independent samples in two distributions on $P_q\times P_q$. The first consists of samples of the form $(a(x), b(x) = a(x)s(x)+e(x))$
where $e(x)$ is chosen from a discretized Gaussian distribution of parameter $\sigma$, and $a(x)$ is uniformly random. The second consists of uniformly random and independent samples from $P_q \times P_q$.

In \cite{CLS}, an attack on PLWE was presented in rings $P_q = F_q[x]/(f(x))$, where $f(1)\equiv 0 \mod q$.

There are also two standard PLWE problems, quoted here from \cite {RLWE for NT}. Let $\mathbb{K}$ be number field of degree $n$ with ring of integers $R$. Let $R^v$ denote the dual of $R$,
$R^v = \{\alpha \in K : Tr(\alpha x) \in \mathbb{Z} \mbox{ for all } x \in R\}$. The standard RLWE problems \cite{on ideal lattices} for a canonical discretized Gaussian are defined as follows.

{\bf Search RLWE Problem}. Let $s \in R^v_q$ be a secret.
The search RLWE problem is to discover s given access to arbitrarily many independent samples of the form $(a, b = as + e)$ where $e$ is chosen from the canonical discretized Gaussian and $a$ is uniformly random.

{\bf Decision RLWE Problem}. Let $s \in R_q$ be a secret.
The decision RLWE problem is to distinguish with non-negligible advantage between the same number of independent samples in two distributions on $R_q\times R^v_q$ . The first consists of samples of the form $(a, b = as+e)$ where $e$ is chosen from the canonical discretized Gaussian and a is uniformly random, and the second consists of uniformly random and independent samples from
$R_q × R^v_q$ .

In ~\cite{attacks on search}, ~\cite{on ideal lattices} the authors give sufficient conditions on the ring so that the ``search-to-decision" reduction for RLWE holds, and also that RLWE instances can be translated into PLWE instances, so that the RLWE decision problem can be reduced to the PLWE decision problem.

\begin{theorem}[Search-to-Decision Reduction for RLWE,~\cite{attacks on search},~\cite{on ideal lattices}]
There exists a randomized, polynomial time reduction from Search-RLWE to Decision-RLWE\@.
\end{theorem}


We investigate the spectral distortion that occurs in the RLWE to PLWE reduction (\emph{spectral distortion}), a question posed in~\cite{RLWE for NT}. Our results include a closed form for the spectral distortion of certain classes of polynomials, and bounds for spectral distortion and related values.

\section{Preliminaries} 

\subsection{Learning with Errors Distributions}
The \emph{RLWE distribution} is parameterized by $(K,s,q,σ)$, where $K$ is a number field, $s$ is some secret, $q$ prime, and $σ$ is the parameter for the error distribution.
\begin{definition}[RLWE Distribution,~\cite{RLWE for NT}] ~\vspace{1mm}\\
For some number field $K$, let ring $R=\mathcal{O}_K$ be its ring of integers. Suppose $q$ to be prime. Then, we define \[R_q≔ R/qR.\] Let $\mathcal{U}_{R_q}$ be the uniform distribution over $R_q$, and let $\mathcal{G}_{σ, R_q}$ be the discrete Gaussian distribution centered at 0 with variance $σ^2$ over $R_q$.
\noindent Let some $s\in R_q$ be the secret. Sample $a$ from the uniform distribution, $a\leftarrow \mathcal{U_{R_q}}$, and the error $e$ from the Gaussian distribution, $e\leftarrow \mathcal{G}_{σ, R_q}$.
\noindent Pairs of the form
\[(a, a⋅ s + e)\]
make up the \textit{RLWE distribution} $\mathcal{L}_{s, G_{σ}}$ over $R_q\times R_q$. For simplicity, we let $c = a⋅ s + e$, and refer to $(a,c)$ as our sample in the future.
\end{definition}

\noindent The \emph{PLWE distribution} is defined similarly; rather than the ring of integers of a number field, the distribution is defined over a polynomial ring. The PLWE distribution is parameterized by $(f, n, s, q, σ)$, where $f\inℤ[x]$ is a monic, irreducible polynomial of degree $n$, $s$ is some secret, $q$ prime, and $σ$ is the parameter of the error distribution.
\begin{definition}[PLWE Distribution,~\cite{RLWE for NT}]~\vspace{1mm}\\
Let $f\in ℤ[x]$ be monic, irreducible of degree $n$. Assume that $f$ splits over $ℤ_q≔ ℤ/qℤ$. Then, we define
\[P\coloneqq ℤ[x]/(f(x)), P_q\coloneqq P/qP.\]

\noindent Let $G_{σ, P}$ be a discretized Gaussian over $P$ spherical in the power basis of $P$ $(1,x,x^2,\ldots, x^{n-1})$. Let $\mathcal{U}_{P_q}$ be the uniform distribution over $P_q$, and let $\mathcal{G}_{σ, P_q}$ be the discrete Gaussian distribution centered at 0 with variance $σ^2$ over $P_q$.

\noindent Let some $s\in P_q$ be the secret. Sample $a$ from the uniform distribution, $a\leftarrow \mathcal{U_{P_q}}$, and the error $e$ from the Gaussian distribution, $e\leftarrow \mathcal{G}_{σ, P_q}$.

\noindent Pairs of the form
\[(a, a⋅ s + e)\]
make up the \textit{PLWE distribution} $\mathcal{L}_{s, G_{σ}}$ over $P_q\times P_q$. Similarly to RLWE, we let $c = a\cdot s + e$, and refer to the samples $(a,c)$.
\end{definition}




\subsection{Spectral Distortion}

In this section, we reference several terms commonly associated with the computation of spectral distortion. 

\begin{definition}
Let $f$ be a monic, irreducible polynomial over $ℤ$ of degree $n$, with some root $\alpha$, and all roots $α_i$. Let $M_f$ be the Vandermonde matrix ${(α_i^{j-1})}_{ij}$. The \emph{Minkowski embedding} of the number field $K=ℚ(α)$ is a function $M:K→ℝ^{r_1}⊗ℂ^{2r_2}$, where every component of $M$ is a field homomorphism, $r_1$ is the number of real roots of $f$, and $2r_2$ is the number of complex roots of $f$.\\
\end{definition}

Let $B$ be the unitary matrix
\[\begin{bmatrix}
I_{r_1\times r_1} & 0 & 0 \\
0 & \frac{\sqrt{2}}{2}I_{r_2×r_2} & \frac{i\sqrt{2}}{2}I_{r_2×r_2} \\
0 & \frac{\sqrt{2}}{2}I_{r_2×r_2} & \frac{-i\sqrt{2}}{2}I_{r_2×r_2} \\
\end{bmatrix}\]
The columns of $B$ give an orthonormal basis under which the Minkowski space is isomorphic to $\mathbb{R}^n$ as an inner product space \cite{provably weak revisited}. Note that the $\sqrt{2}$ factor ensures this $B$ is unitary.
Because $B$ is unitary, $B^{-1}=B^†$.

\begin{remark}

We note here that $B^†M_f=B^{-1}M_f$ is the transpose of the real matrix
\[\begin{bmatrix}
σ_1(1) & ⋯ & σ_{r_1}(1) & \sqrt{2}ℜ(σ_{r_1+1}(1)) & ⋯ & \sqrt{2}ℜ(σ_{r_1+r_2}(1)) & \sqrt{2}ℑ(σ_{r_1+1}(1)) & ⋯ & \sqrt{2}ℑ(σ_{r_1+r_2}(1)) \\
σ_1(α) & ⋯ & σ_{r_1}(α) & \sqrt{2}ℜ(σ_{r_1+1}(α)) & ⋯ & \sqrt{2}ℜ(σ_{r_1+r_2}(α)) & \sqrt{2}ℑ(σ_{r_1+1}(α)) & ⋯ & \sqrt{2}ℑ(σ_{r_1+r_2}(α)) \\
⋮ & & ⋮ &  ⋮& & ⋮  & ⋮ & & ⋮\\
σ_1(α) & ⋯ & σ_{r_1}(α) & \sqrt{2}ℜ(σ_{r_1+1}(α)) & ⋯ & \sqrt{2}ℜ(σ_{r_1+r_2}(α)) &  \sqrt{2}ℑ(σ_{r_1+1}(α)) & ⋯ & \sqrt{2}ℑ(σ_{r_1+r_2}(α)) \\
\end{bmatrix}\]

We have
\[{(B^†M_f)}^†(B^†M_f)=M_f^†BB^†M_f=M_f^†M_f\]
Therefore, we may implicitly compute using $B^†M$ instead of $M$. We will use this fact in several of the proofs in this paper.

Because $B^†M_f$ is real, ${(B^†M_f)}^†(B^†M_f)$ is real, and ${(B^†M_f)}^†(B^†M_f)$ is conjugate transpose symmetric, so $M^†M$ is a real, symmetric matrix.

\end{remark}

\begin{definition}
The \emph{spectral norm} $\norm{M}_2$ is the measure of the distortion between RLWE and PLWE for a specific polynomial $f$, given by the largest singular value of $M_f^†M_f$ \cite{provably weak}. 
The normalized spectral norm, or \emph{spectral distortion}, provides another measure of distortion that is a convenient quantity in reductions from PLWE to RLWE\@.
The spectral distortion is defined by 
\[\SD(f) = \frac{\norm{M_f^{-1}}_2}{|\det M^{-1}|^{\frac{1}{n}}} = \frac{\frac{1}{σ_{\min}(M_f)}}{\frac{1}{|\det M_f|^{\frac{1}{n}}}} = \frac{|\det M_f|^{\frac{1}{n}}}{σ_{\min}(M_f)}\]
\end{definition}


\section{Cyclotomic Polynomials and Bounds on Spectral Distortion}

We first consider the case that $f$ is a cyclotomic polynomial, the current class of candidates for lattice-based homomorphic encryption with ideal lattices \cite{RLWE for NT}. In addition, cyclotomic polynomials tend to have a comparatively smaller spectral norm than general polynomials.
In this case, the $M^†M$ matrix has a convenient formula, from which its eigenvalues can be determined easily in some cases.

\begin{theorem}\label{1}
  Let $n=p_1^{k_1}⋯p_{ω(n)}^{k_{ω(n)}}$, for primes $p_i$ and $k_i\in\mathbb{N}$. Then, the $M_f^†M_f$ matrix is of the following form:
\[{(M_f^†M_f)}_{ij}=\begin{cases}
φ(n) & \text{if~}~ i=j \\
0 & \text{if~}~ \frac{n}{\rad\parens{n}}∤ i-j \\
{(-1)}^{ω(n)+ω(d)}\parens{\frac{n}{\rad(n)}}φ\parens{\rad\parens{d}} & \text{if~}~ \frac{n}{\rad(n)}\mid i-j
\end{cases}\]
where $d = \gcd\left(\frac{i-j}{n/\rad(n)}, n\right)$
\end{theorem}
\begin{proof}

Let $c_1,…,c_{φ(n)}$ be the integers coprime to $n$, up to $n$. Then, we label the roots of $f$, the primitive $n$-th roots of unity, as  $ζ_n^{c_1},…,ζ_n^{c_{φ(n)}}$.
By properties of $n$-th roots of unity, we know that $ζ_n^{c_l}$ and $ζ_n^{c_{\varphi(n)+1-l}}$ are complex conjugates.

Then, we note that the $j$-th row of $M^†$ looks like
\[\begin{bmatrix}
    \sqrt{2}~\Re\parens{ζ_n^{jc_1}} & ⋯ & \sqrt{2}~\Re\parens{ζ_n^{jc_{φ(n)/2}}} & \sqrt{2}~\Im\parens{ζ_n^{jc_1}} & ⋯  & \sqrt{2}~\Im\parens{ζ_n^{jc_{φ(n)/2}}}
  \end{bmatrix}
\]
where $\Re(ζ_n^{c_l})=\cos(2πc_l/n)$ and $ℑ(ζ_n^{c_l})=\sin(2πc_l/n)$.

\begin{equation*}
\begin{split}
  {(M_f^†M_f)}_{ij}
&=2∑_{l=1}^{φ(n)/2} \Big(\cos(2πic_l/n)\cos(2πjc_l/n)+\sin(2πic_l/n)\sin(2πjc_l/n)\Big) \\
&=2∑_{l=1}^{φ(n)/2} \cos(2π c_l(i-j)/n) 
=2∑_{l=1}^{φ(n)/2} ℜζ_{n}^{c_l(i-j)} \\
&=∑_{l=1}^{φ(n)/2} \parens{ℜζ_{n}^{(i-j)c_l}+ℜζ_{n}^{-(i-j)c_l}} =∑_{l=1}^{φ(n)} ℜζ_{n}^{(i-j)c_l}
\end{split}
\end{equation*}

Let $g_l$ iterate through the $n-φ(n)$ integers not coprime to $n$. If $i-j=0$, then we see that ${(M_f^†M_f)}_{ij}=φ(n)$.
If $i-j≠0$, then we have
\[∑_{l=1}^{φ(n)} ζ_{n}^{(i-j)c_l}+∑_{l=1}^{n-φ(n)} ζ_{n}^{(i-j)g_l} = ∑_{l=0}^{n-1}ζ_{n}^{(i-j)l} = 0\implies {(M_f^†M_f)}_{ij}=-\Re∑_{g_l}ζ_{n}^{(i-j)g_l}\]

The next part of the proof uses inclusion-exclusion on the prime factors of $n$ to count all roots with a nontrivial common factor to $n$ (or all roots not coprime to $n$).
Let $p_1,…,p_{ω(n)}$ be the prime factors of $n$ where  $ω(n)$ denotes the number of all distinct prime factors of $n$.
For the last term, there is just one possible set of $ω(n)$ unique prime factors.

\[
  \begin{split}
    -∑_{g_l}ζ_{n}^{(i-j)g_l}
    &=-∑_{k=1}^{ω(n)} ∑_{t=0}^{n/p_k-1} ζ_n^{(i-j)tp_k} 
    +∑_{k<l}^{ω(n)} ∑_{t=0}^{n/(p_kp_l)-1} ζ_{n}^{(i-j)tp_kp_l}
    +…+{(-1)}^{ω(n)} ∑_{t=0}^{n/(\rad(n))-1} ζ_n^{(i-j)t\rad(n)} \\
    &=∑_{k=1}^{ω(n)} {(-1)}^k ∑_{p_{l_1}<⋯ <p_{l_k}} ∑_{t=0}^{n/∏_s p_{l_s} -1} ζ_{n}^{(i-j)t∏_s p_{l_s}} \\
\end{split}
\]

We observe

\[∑_{t=0}^{n/∏_s p_{l_s} -1} ζ_{n}^{(i-j)t∏_s p_{l_s}}
  =∑_{t=0}^{n/∏_s p_{l_s} -1} ζ_{n/∏_s p_{l_s}}^{(i-j)t}
  = \begin{cases}\frac{n}{∏p_{l_s}} &  \frac{n}{∏_s p_{l_s}}\mid (i-j) \\ 0 &  \frac{n}{∏_s p_{l_s}}∤ (i-j)\end{cases}
  = \frac{n}{\rad(n)} \begin{cases}\frac{\rad n}{∏ p_{l_s}} &  \frac{n}{∏_s p_{l_s}}\mid (i-j) \\ 0 &  \frac{n}{∏_s p_{l_s}}∤ (i-j)\end{cases}
\]

Let $\Pi_r p_{l_r} = \frac{\rad n}{∏ p_{l_s}}$ be the complement set of $ω(n)-k$ primes where $\rad(n)$ denotes the product of all distinct prime factors of $n$. Then,

\[
  \begin{split}
    &=∑_{k=1}^{ω(n)} {(-1)}^k ∑_{p_{l_1}<⋯ <p_{l_k}} ∑_{t=0}^{n/∏_s p_{l_s} -1} ζ_{n}^{(i-j)t∏_s p_{l_s}} \\
    &=\frac{n}{\rad(n)} ∑_{k=1}^{ω(n)} {(-1)}^k ∑_{p_{l_1}<⋯ <p_{l_{ω(n)-k}}} \begin{cases}∏_r p_{l_r} & \text{if~}~ \frac{n∏_r p_{l_r}}{\rad n}\mid (i-j) \\ 0 & \text{if~}~ \frac{n∏_r p_{l_r}}{\rad n} ∤ (i-j)\end{cases} \\
  \end{split}
  \]

  We see that if $\frac{n}{\rad n}∤ i-j$, then $\frac{n∏_r p_{l_r}}{\rad n} ∤ (i-j)$, and the above summations are all zero.
If $\frac{n}{\rad n}\mid i-j$, then we can factor out $\frac{n}{\rad n}\mid i-j$ from our cases to get

\[
  {M^†M}_{ij}
  =\frac{n}{\rad(n)} ∑_{k=1}^{ω(n)} {(-1)}^k ∑_{p_{l_1}<⋯ <p_{l_{ω(n)-k}}} \begin{cases}∏_r p_{l_r} & \text{if~}~ ∏_r p_{l_r} \mid \frac{i-j}{n/\rad n} \\ 0 & \text{if~}~ ∏_r p_{l_r} ∤ \frac{i-j}{n/\rad n} \end{cases} \\
\]

  Note that since $\frac{n}{\rad n}\mid i-j$, then $n\mid (i-j)\rad(n)$, and $ζ_n^{(i-j)t\rad(n)} = 1$. So, the last term of our summation is $${(-1)}^{ω(n)} ∑_{t=0}^{n/(\rad(n))-1} ζ_n^{(i-j)t\rad(n)} = {(-1)}^{ω(n)}\frac{n}{\rad(n)}$$
  
  If there are no primes $p_l$ such that $p\mid \frac{i-j}{n/\rad n}$, then all of the other summations are zero, and $M^† M = {(-1)}^{ω(n)}\frac{n}{\rad(n)}$.
  Otherwise, let $d=\gcd(\frac{i-j}{n/\rad(n)}, n)$. 
  There exist $k = \omega(d)$ primes $q_1,…,q_k$ that do divide $(i-j)/(n/\rad(n))$ and $n$. 
  \vspace{3mm}\\

Let $S = {q_1, q_2, …, q_k}$ be the set of all such primes. Since $\forall q\in S, q \mid (i-j)/(n/\rad(n))$, we know that for any subset $S_1\subset S$, $∏_{q\in S_1}\mid (i-j)/(n/\rad(n))$.

Moreover, if any product contains primes $p$ such that $p\not\in S$, then that product cannot divide $(i-j)/(n/\rad(n))$, as $p\nmid (i-j)/(n/\rad(n))$. \\
Thus, every nonzero term in our summation corresponds exactly to the product of elements in $S_1,\forall S_1\subset S$, and we can rewrite our expression as below.

\vspace{3mm}

Let $c = {(-1)}^{ω(n)-ω(d)}$. We can factor the summation as follows:
\begin{equation*}
    \begin{split}
      {(M_f^† M_f)}_{ij}
&=c⋅\frac{n}{\rad(n)}\Big(q_1⋯ q_k - ∑_{q_{f_1} …  q_{f_{k-1}}\in S}q_{f_1}⋯ q_{f_{k-1}} + … + (-1)^{k-1}∑_{q\in S}q + (-1)^k\Big)\\
&=c⋅\frac{n}{\rad(n)}(q_k-1)\Big(q_1⋯ q_{k-1} - ∑_{q_{f_1} … q_{f_{k-2}}\in S\setminus q_k}q_{f_1}⋯ q_{f_{k-2}} + … +  (-1)^{k-1} \Big)\\
&\vdots\\
&=c\frac{n}{\rad(n)}(q_k-1)(q_{k-1}-1)⋯ (q_2-1)(q_1-1)=c\parens{\frac{n}{\rad n}}∏_{q\in S}φ(q) \\
&=c\parens{\frac{n}{\rad n}}φ\parens{\rad\parens{\gcd\parens{\frac{i-j}{n/\rad\parens{n}},n}}}
\end{split}
\end{equation*}

We get the desired result
\[{(M_f^† M_f)}_{ij} = (-1)^{ω(n)-ω(d)}\parens{\frac{n}{\rad(n)}}φ\parens{\rad d}\]
\end{proof}

\begin{corollary}
Let $f=Φ_n$ be $n^{th}$ cyclotomic polynomial. The $M_f^†M_f$ matrix for $f$ is of the form:
$$M_{Φ_n}^†M_{Φ_n}=\parens{\frac{n}{\rad{n}}}M_{Φ_{\rad(n)}}^† M_{Φ_{\rad(n)}}⊗I_{\frac{n}{\rad{n}}}$$
\end{corollary}

\begin{remark}
Let the eigenvalues of $M_{\rad(n)}^† M_{\rad(n)}$ be $λ_1,…λ_{φ(\rad(n))}$.
This implies that the eigenvalues of $M_n^† M_n$ are $\frac{n}{\rad{n}}λ_1,…,\frac{n}{\rad{n}}λ_{φ(\rad(n))}$ with multiplicity $\frac{n}{\rad{n}}$.
In particular, for a prime $p$, $M^† M=pI_{φ(p)}-1_{φ(p)}$. 
Also, in particular, for any number $n$ with prime factor $p$, $M_{Φ_{np}}^† M_{Φ_{np}}=pM_{Φ_n}⊗I_p$.
\end{remark}

\begin{remark}
Note that $M_{Φ_n}^† M_{Φ_n}$ forms a symmetric Toeplitz matrix. \footnote{A Toeplitz matrix, or a diagonal-constant matrix, is a matrix $A$ such that $A_{i,j}=A_{i+1,j+1}$} 

We can also describe the $M_{Φ_n}^† M_{Φ_n}$ matrix's construction as follows:
\begin{itemize}
\item Let $t=p_1\dotsm p_s$ be a squarefree integer. Then the matrix $M^† M$ for $Φₜ$ is given by the symmetric Toeplitz matrix generated by the vector $v$, where $v$ is constructed as follows:
\begin{enumerate}
    \item Let $v$ be a constant vector of value $(-1)^s$ of length $φ(t)$, indexed by $i$ from 0 to $φ(t)-1$.
    \item For all $i$, if $p_j$ divides $i$, then let $v[i]←-φ(p_j)*v[i]$
\end{enumerate}
\item Let $n=p_1^{k_1}\dotsm p_s^{k_s}$ be an arbitrary integer and $L$ be the Toeplitz matrix of $s$ as constructed above. Then the matrix $M_f^† M_f$ for $n$ is given by $\frac{n}{\rad n}L⊗I_{\frac{n}{\rad n}}$ where $I_q$ is the identity matrix of size $q$.
\item Equivalently, the matrix for $n$ can be given by 
\[
\parens{n/s}\parens{\circ_{i=1}^s 
\parens{\sqbracks{\begin{array}{c|c}1_{φ(s)} & 0_{⌈φ(s)/p⌉p}\end{array}}*\parens{p*1_{⌈φ(s)/p⌉}⊗I_p}*\sqbracks{\begin{array}{c} 1_{φ(s)} \\ \hline 0_{⌈φ(s)/p⌉p}\end{array}}-1_{φ(s)}}}⊗I_{n/s}\]
where $\circ$ denotes the Hadamard, or entrywise, product.
\end{itemize}

\end{remark}

\begin{example}
For $f=Φ_{15}$, we have a symmetric Toeplitz matrix

  \center
  \[\begin{bmatrix}
8 & 1 & 1 & -2 & 1 & -4 & -2 & 1 \\
1 & 8 & 1 & 1 & -2 & 1 & -4 & -2 \\
1 & 1 & 8 & 1 & 1 & -2 & 1 & -4 \\
-2 & 1 & 1 & 8 & 1 & 1 & -2 & 1\\
1 & -2 & 1 & 1 & 8 & 1 & 1 & -2 \\
-4 & 1 & -2 & 1 & 1 & 8 & 1 & 1 \\
-2 & -4 & 1 & -2 & 1 & 1 & 8 & 1\\
1 & -2 & -4 & 1 & -2 & 1 & 1 & 8
\end{bmatrix}
\]
\end{example}

We can use this rich structure to derive more specific properties of spectral distortion for cyclotomic polynomials.The following theorem shows that the spectral distortion of the $n$th cyclotomic polynomial depends only on the radical of $n$.

\begin{corollary}\label{5}
\[\SD(Φ_n)=\SD(Φ_{\rad n})\]
\end{corollary}
\begin{proof}
Let $n≥1$.
Let $p$ be a prime that divides $n$.
We show $\SD(Φ_n)=\SD(Φ_{np})$.
For cyclotomic polynomials, $|\Disc(Φ_n)|=\frac{n^{φ(n)}}{∏_{p|n}(p^{φ(n)/p-1})}$.
\[\det(M_{Φ_{np}})^{1/φ(np)}=\sqrt{\frac{(np)^{φ(np)}}{∏_{p|n}(p^{φ(np)/(p-1)})}}^{1/φ(np)}=\sqrt{\frac{np}{∏_{p|n}(p^{1/(p-1)})}}\]
\[\det(M_{Φ_n})^{1/φ(n)}=\sqrt{\frac{n^{φ(n)}}{∏_{p|n}(p^{φ(n)/p-1})}}^{1/φ(n)}=\sqrt{\frac{n}{∏_{p|n}(p^{1/(p-1}))}}\]
$$\implies \det(M_{Φ_{np}})^{1/φ(np)} = \sqrt{p}\det(M_{Φ_n})^{1/φ(n)}$$
We see in Theorem \ref{1} that the largest eigenvalue of $M_{Φ_{np}}$ increases by a factor of $p$, so $‖M_{Φ_{np}}‖ = \sqrt{p}‖M_{Φ_{n}}‖$. Thus, we have
$$\SD(\Phi_{np}) = \frac{|\det M_{Φ_{np}}|^{\frac{1}{\varphi(np)}}}{σ_{\min}(M_{\Phi_{np}})} = \frac{\sqrt{p}|\det M_{Φ_{n}}|^{\frac{1}{\varphi(n)}}}{\sqrt{p}\cdotσ_{\min}(M_{\Phi_{n}})} = \SD(\Phi_{n})$$
\end{proof}

\begin{theorem}\label{2}
The eigenvalues of $M_{Φ_p}^† M_{Φ_p}$ for prime $p$ are 1 with multiplicity 1 and $p$ with multiplicity $p-2$.
\end{theorem}
\begin{proof}
By~\ref{1}, $M_{Φ_p}^† M_{Φ_p}$ is a circulant matrix with row entries $c₀=p-1,c_1=⋯=c_{p-2}=-1$.
By well-known properties of circulant matrix eigenvalues, for $0≤j<p-2$, the eigenvalues of $M_{Φ_p}^† M_{Φ_p}$ are of the form
\begin{equation*}
    \begin{split}
        λ_j
        &=c₀+∑_{k=1}^{p-2}c_{p-1-k}ζ^{jk} \\
        &=(p-1)-∑_{k=1}^{p-2}ζ^{jk} \\
    \end{split}
\end{equation*}
If $j=0$, then
$$(p-1)-∑_{k=1}^{p-2}ζ^{jk} = (p-1)-∑_{k=1}^{p-2}1 = (p-1) - (p-2) = 1$$
For the other $p-2$ cases, $j≠0$, and 
$$(p-1)-∑_{k=1}^{p-2}ζ^{jk} = (p-1) + ζ^0 - ζ^0 -∑_{k=1}^{p-2}ζ^{jk}= (p-1) + 1 -∑_{k=0}^{p-2}ζ^{jk} = (p-1) + 1 - 0 = p$$
\end{proof}

\begin{corollary}
For prime $p$,
$$\SD(Φ_{p}) = p^{\frac{p-2}{2(p-1)}}$$
\end{corollary}
\begin{proof}
For cyclotomic polynomials, $|\Disc(Φ_n)|=\frac{n^{φ(n)}}{∏_{p|n}(p^{φ(n)/p-1})}$.
\[\det(M_{p})^{1/(p-1)}=\sqrt{\frac{(p)^{(p-1)}}{(p^{(p-1)/(p-1)})}}^{1/(p-1)}= \sqrt{\frac{p^{p-1}}{p}}^{1/(p-1)} = p^{\frac{p-2}{2(p-1)}}\]
We know that $Det(M^{-1}) = \big(Det(M)^{1/(p-1)}\big)^{-1} = p^{-\frac{p-2}{2(p-1)}}$. 
We know also from~\ref{2} that the smallest eigenvalue of $M_{Φ_p}^†M_{Φ_p}$ for prime $p$ is 1. 
So, $$\norm{M_{Φ_p}^{-1}} = \frac{1}{σ_{min}(M_{Φ_p})} = 1$$
$$\SD(Φ_n) = \frac{\norm{M_{Φ_p}^{-1}}_2}{|\det(M_{Φ_p}^{-1})^{1/(p-1)}|} = \frac{1}{p^{-\frac{p-2}{2(p-1)}}} = p^{\frac{p-2}{2(p-1)}}$$
\end{proof}

\begin{lemma}\label{cyc_2n}
The $M^{\dagger}_fM_f$ matrix for $f = \Phi_{2n}$, $2\nmid n$, is of the form:
$$\left(M^{\dagger}_{\Phi_{2n}}M_{\Phi_{2n}}\right)_{ij} = (-1)^{i+j}\left(M^{\dagger}_{\Phi_n}M_{\Phi}\right)_{ij}$$
\end{lemma}
\begin{proof}
Note that since $2\nmid n$, $\phi(2n) = 2\left(1-\frac{1}{2}\right)\phi(n) = \phi(n)$, and $\frac{2n}{\rad(2n)} = \frac{2n}{2\rad(n)} = \frac{n}{\rad(n)}$. We need to check each case given in~\ref{1}.\vspace{3mm}\\
\textbf{Case 1:} $i = j$\vspace{3mm}\\
In this case, $$\left(M^{\dagger}_{\Phi_{2n}}M_{\Phi_{2n}}\right)_{ij} = \phi(2n) = 2\left(1-\frac{1}{2}\right)\phi(n) = \phi(n)$$
$$= \left(M^{\dagger}_{\Phi_{n}}M_{\Phi_{n}}\right)_{ij} = (-1)^{i+j}\left(M^{\dagger}_{\Phi_{n}}M_{\Phi_{n}}\right)_{ij}$$
as $2\mid (i+j)$.\vspace{3mm}\\
\textbf{Case 2:} $\frac{2n}{\rad(2n)}\nmid (i-j)$\vspace{3mm}\\
Since $\frac{2n}{\rad(2n)} = \frac{n}{\rad(n)}$, then $\frac{n}{\rad(n)}\nmid (i-j)$, and 
$$\left(M^{\dagger}_{\Phi_{2n}}M_{\Phi_{2n}}\right)_{ij} = 0 = (-1)^{i+j}\left(M^{\dagger}_{\Phi_{n}}M_{\Phi_{n}}\right)_{ij}$$\vspace{3mm}\\
\textbf{Case 3:} $\frac{2n}{\rad(2n)}\mid (i-j)$\vspace{3mm}\\
Recall that $ω(n)$ is the number of distinct prime factors of $n$. Note that $ω(2n) = ω(n) + 1$, as $2\nmid n$.\vspace{3mm}\\ Consider when $2\nmid (i-j)$. Then, $2\nmid \frac{i-j}{n/\rad(n)}$, and
$$\gcd\parens{\frac{i-j}{2n/\rad(2n)}, 2n} = \gcd\parens{\frac{i-j}{n/\rad(n)}, 2n} = \gcd\parens{\frac{i-j}{n/\rad(n)}, n}$$
so $d_{2n} = d_n$. Thus,
$$\left(M^{\dagger}_{\Phi_{2n}}M_{\Phi_{2n}}\right)_{ij} = (-1)^{s_n + \omega(d_n) + 1}\left(\frac{n}{\rad(n)}\right)\phi(\rad(d_n)) $$$$= -\left(M^{\dagger}_{\Phi_{n}}M_{\Phi_{n}}\right)_{ij} = (-1)^{i+j}\left(M^{\dagger}_{\Phi_{n}}M_{\Phi_{n}}\right)_{ij}$$\vspace{3mm}\\
Consider now when $2\mid(i-j)$. Then, $2\mid \frac{i-j}{n/\rad(n)}$, and 
$$\gcd\parens{\frac{i-j}{2n/\rad(2n)}, 2n} = 2\gcd\parens{\frac{i-j}{n/\rad(n)}, n}$$
so $d_{2n} = 2d_n$, and $\omega(d_{2n}) = \omega(d_{2n})+1$. Thus,
$$\left(M^{\dagger}_{\Phi_{2n}}M_{\Phi_{2n}}\right)_{ij} = (-1)^{s_n + \omega(d_n) + 1 + 1}\left(\frac{n}{\rad(n)}\right)\phi(\rad(d_n))$$
$$= \left(M^{\dagger}_{\Phi_{n}}M_{\Phi_{n}}\right)_{ij} = (-1)^{i+j}\left(M^{\dagger}_{\Phi_{n}}M_{\Phi_{n}}\right)_{ij}$$
\end{proof}

\begin{lemma}\label{cyc_2n_lemma}
Let $A$ be a matrix.
The matrix $((-1)^{i+j} A_{ij})_{ij}$ has the same eigenvalues as $A$.
\end{lemma}
\begin{proof}
The eigenvalues of $A$ are defined by the characteristic equation $\det(λI-A)$.

By the Leibniz formula for determinants,
\[\det(λI-((-1)^{i+j}A_{ij})_{ij})=∑_{σ} (-1)^{σ} ∏_{i} (λI - (-1)^{i+σ(i)} A_{iσ(i)})\]

Taking out the identity permutation, we have
\[∏_{i} (λI-A_{ii}) + ∑_{σ/i} (-1)^{σ} ∏_{i} (-1)^{i+σ(i)}A_{iσ(i)}\]

Because
\[∏_i (-1)^{i + σ(i)} = ∏_i (-1)^i ∏_i (-1)^{σ(i)} = (-1^{\frac{φ(n)(φ(n)+1)}{2}})^2 = 1\]

We have
\[∏_{i} (λI-A_{ii}) + ∑_{σ/i} (-1)^{σ} ∏_{i} A_{iσ(i)}=\det(λI-A)\]
\end{proof}

\begin{theorem}\label{3}
Let $n∈ℕ$ be odd.
The eigenvalues of $M_{Φ_{2n}}^† M_{Φ_{2n}}$ are the same as the  eigenvalues of $M_{Φ_{n}}^† M_{Φ_{n}}$.
\end{theorem}
\begin{proof}
From Lemma \ref{cyc_2n}, we know that
$(M_{Φ_{2n}}^† M_{Φ_{2n}})_{ij}=(-1)^{i+j}(M_{Φ_{n}}^† M_{Φ_{n}})_{ij}$.
The proof then follows directly from the above lemma \ref{cyc_2n_lemma}.
\end{proof}

\begin{corollary}\label{4}
For odd $n$,
$$\SD(Φ_{2n}) = \SD(Φ_n)$$
\end{corollary}
\begin{proof}
First we look at the denominator, $\det(M_{Φ_n})^{1/φ(n)}$:
$$\det(M_{Φ_n})^{1/φ(n)}=\sqrt{\frac{n^{φ(n)}}{∏_{p|n}(p^{φ(n)/p-1})}}^{1/φ(n)}=\sqrt{\frac{n}{∏_{p|n}(p^{1/(p-1}))}}$$
$$\det(M_{Φ_{2n}})^{1/φ(2n)}=\sqrt{\frac{(2n)^{φ(2n)}}{∏_{p|(2n)}(p^{φ(2n)/p-1})}}^{1/φ(2n)}=\sqrt{\frac{2n}{∏_{p|(2n)}(p^{1/(p-1}))}}$$
$$= \sqrt{\frac{2n}{2\left(∏_{p|n}(p^{1/(p-1}))\right)}} = \sqrt{\frac{n}{∏_{p|n}(p^{1/(p-1}))}}$$
$$\implies \det(M_{Φ_n})^{1/φ(n)} = \det(M_{Φ_{2n}})^{1/φ(2n)}$$
From Theorem~\ref{3}, we know that the eigenvalues of $M_{Φ_{2n}}^†M$ are the same as those of $M_{Φ_n}^†M$, and therefore the spectral norm for $2n$ and $n$ are the same. It follows that $\SD(Φ_{2n}) = \SD(Φ_n)$.
\end{proof}

\subsection{Non-Cyclotomic Polynomials}

We now turn to results that encompass non-cyclotomic polynomials.

\begin{theorem}\label{6}
Let $h(x)$ be a monic, irreducible polynomial over $ℤ$.
Let $f(x)=h(x^k)$.
Let $α_t$ be the roots of $h(x)$.
\[
 (M_{f}^†M_{f})_{ij}
        = \begin{cases}
        k \parens{∑_{\text{real } α_t} α_t^{(i+j)/k} + ∑_{\text{non-real }α_t} α_t^{i/k}
        \overline{α_t}^{j/k}} & \text{if~}~ k \mid i-j \\
        0 & \text{if~}~ k \nmid i-j
        \end{cases}
\]
\end{theorem}

\begin{proof}
\begin{equation*}
    \begin{split}
        (M_{f}^†M_{f})_{ij}
        &=∑_{α_t} ∑_{s=0}^{k-1} \parens{ζ_k^s α_t^{1/k}}^i \overline{\parens{ζ_k^s α_t^{1/k}}^j} \\
        &=∑_{α_t} α_t^{i/k} \overline{α_t^{j/k}} ∑_{s=0}^{k-1} ζ_k^{s\parens{i-j}} \\
        &=\begin{cases} k ∑_{α_t} \parens{α_t^{i/k}\overline{α_t}^{j/k}} & \text{if~}~ i-j=0\mod k \\ 0 & \text{if~}~ i-j\neq0\mod k \end{cases} \\
    \end{split}
\end{equation*}
If $i=j\mod k$, then
\[      (M_{f}^†{M_{f}})_{ij}
        = k \parens{∑_{\text{real } α_t} α_{t}^{(i+j)/k} + ∑_{\text{non-real } α_t} α_{t}^{i/k}\overline{α_t}^{j/k}} \]
\end{proof}

\begin{corollary}
Let $s=i\mod k$.
\[(M_f^† M_f)_{ij}=\begin{cases} k{(M_h^† M_h)}_{i'j'}\: h(0)^{s/k} & \text{if~}~k\mid i-j \\ 0 & \text{if~}~k \nmid i-j\end{cases}\]
and
\[M_{f}^† M_{f}=M_{h}^†M_{h}⊗
\begin{bmatrix}
h(0)^{0/k} & 0          & 0 & \cdots &&&&\\
0 & h(0)^{1/k}        & 0 & 0 & \cdots &&& \\
0 & 0 &\ddots        &&&&&  \\
\vdots & 0 &&h(0)^{k-1/k}    &&&&   \\
&\vdots&&&h(0)^{0/k}     &&&    \\
&&&&&h(0)^{1/k}    &&     \\
&&&&&&\ddots    &      \\
&&&&&&&h(0)^{k-1/k}
\end{bmatrix}\]
\end{corollary}

\begin{proof}
When $i≡ j\mod k$, we have $i = i'k + s$ and $j = j'k + s$ for some $s≤ k$, $i',j'\in ℤ$. Then, 
\begin{equation*}
    \begin{split}
        k\parens{∑_{α∈ℝ} α^{((i'k+s)+(j'k+s))/k} + ∑_{α\not ∈ ℝ} α^{(i'k+s)/k}\overline{α}^{(j'k+s)/k}}
        &= k\parens{∑_{α∈ℝ} α^{(i'+j'+2s)/k} + ∑_{α ∉ ℝ} α^{s/k}\overline{α}^{s/k} α^{i'}\overline{α}^{j'}} \\
        &= k\parens{∑_{α∈ℝ} α^{(i'+j'+2s)/k} + ∑_{α ∉ ℝ} |α|^{2s/k} α^{i'}\overline{α}^{j'}}
    \end{split}
\end{equation*}
If $h(x)$ is a quadratic polynomial with negative discriminant, then $|α|^{2s/k}=h(0)^{s/k}$, so we have the listed results.
\end{proof}

\begin{corollary}
Let $h(x)=x^2+bx+c$ have negative discriminant.
\[M_h^†M_h=\begin{bmatrix}2 & -b \\ -b & 2c\end{bmatrix}\]
$M_h^†M_h$ has characteristic polynomial
\[(λ-2)(λ-c)-b^2=λ^2-(2+2c)λ-b^2\]
And eigenvalues
\[1+c ± \sqrt{b^2+c^2+2c+1}\]
Therefore, we can calculate the eigenvalues and therefore spectral norm of $M_f$ for all $f(x)=h(x^k)$.
\end{corollary}

\subsection{Bounds on Spectral Distortion}
In \cite{a lower bound}, Hong and Pan derive a lower bound on the smallest singular value of general matrices $A$:

\[σ_{\min}(A) ≥ \parens{\frac{n-1}{n}}^{(n-1)/2} |\det(A)| \max\left\{\frac{r_{\min}(A)}{∏_{i=0}^nr_i(A)}, \frac{c_{\min}(A)}{∏_{i=0}^nc_i(A)}\right\}\]

where $r_i$ is the $L^2$ norm of the $i$th row, and $c_i$ is the $L^2$ norm of the $i$th column.\\
We use this lower bound to create an upper bound for general spectral distortion:
\begin{theorem}\label{7}
Let $r_i$ be the $L^2$ norm of the $i$th row of $M_f$, and $c_i$ be the $L^2$ norm of the $i$th column of $M_f$. 
For a polynomial $f$ of degree $n$,
$$\SD(f) ≤ \parens{\frac{n}{n-1}}^{(n-1)/2} |\det(M_f)|^{\frac{1-n}{n}}
\max\left\{\frac{r_{\min}(M_f)}{∏_{i=0}^nr_i(M_f)}, \frac{c_{\min}(M_f)}{∏_{i=0}^nc_i(M_f)}\right\}$$
\end{theorem}
\begin{proof}
$$\SD(f) = \frac{\norm{M_f^{-1}}_2}{|\det M^{-1}|^{\frac{1}{n}}} = \frac{\frac{1}{σ_{\min}(M_f)}}{\frac{1}{|\det M_f|^{\frac{1}{n}}}} = \frac{|\det M_f|^{\frac{1}{n}}}{σ_{\min}(M_f)}$$
$$\SD(f) ≤ \frac{|\det M_f|^{\frac{1}{n}}}{\parens{\frac{n-1}{n}}^{(n-1)/2} |\det(M_f)| \frac{r_{\min}(M_f)}{∏_{i=0}^n(M_f)}} = (\frac{n}{n-1})^{(n-1)/2} |\det(M_f)|^{\frac{1-n}{n}}
\frac{∏_{i=0}^n(M_f)}{r_{\min}(M_f)}$$
$$\implies \SD(f) ≤ \parens{\frac{n}{n-1}}^{(n-1)/2} |\det(M_f)|^{\frac{1-n}{n}}
\frac{∏_{i=0}^nr_i(M_f)}{r_{\min}(M_f)}$$
\end{proof}

\noindent
Similarly, in \cite{note}, Yu and Gu presented another lower bound on the minimum singular value based on the Frobenius norm. 
With the Frobenius norm defined as $$\norm{A}_F^2 = ∑_{i=1}^n ∑_{j=1}^n |a_{ij}|^2$$
The minimum singular value of matrix $A$ is bounded as follows:
$$σ_{\min}(A)\geq |\det A| \left(\frac{n-1}{\norm{A}_F^2}\right)^{\frac{n-1}{2}}$$
We use this now to propose another bound on spectral distortion.

\begin{theorem}\label{9}
For a polynomial $f$ of degree $n$,
$$\SD(f)_\leq \parens{\frac{\norm{M_F^2}}{n-1}}^{\frac{n-1}{2}}|\det{M_f}|^{\frac{1-n}{n}}$$
\end{theorem}
\begin{proof}
$$\SD(f) = \frac{\norm{M_f^{-1}}_2}{|\det M^{-1}|^{\frac{1}{n}}} = \frac{\frac{1}{σ_{\min}(M_f)}}{\frac{1}{|\det M_f|^{\frac{1}{n}}}} = \frac{|\det M_f|^{\frac{1}{n}}}{σ_{\min}(M_f)}$$
$$\SD(f)\leq \frac{|\det{M_f}|^{\frac{1}{n}}}{\parens{\frac{n-1}{\norm{M_f}_F^2}}^{\frac{n-1}{2}}|\det{M_f}|}$$
This implies
\[\SD(f)_\leq \parens{\frac{\norm{M_F^2}}{n-1}}^{\frac{n-1}{2}}|\det{M_f}|^{\frac{1-n}{n}}\]
\end{proof}


\section{Conclusion}

In this paper, we showed that the $M_f^{\dagger}M_f$ matrix from which the spectral distortion is derived has a convenient formula with special properties for the case of a cyclotomic polynomial $f$.
Moreover, we derived mild generalizations of these properties for non-cyclotomic polynomials. Finally, we found bounds on the eigenvalues of this matrix for the general case, as well as bounds on the spectral distortion in the cyclotomic case.


\begin{thebibliography}{1}


\bibitem {A2} M. R. Albrecht and A. Deo, \emph{Large modulus ring-LWE $\geq$ module-LWE},  {\bf  ASIACRYPT 2017} ,  Vol. 10624, (2017) 267-- 296 .

\bibitem {B} 
A. Banerjee,, C. Peikert and A. Rosen, \emph{Pseudorandom functions and lattices}, {\bf EUROCRYPT 2012}, Lecture Notes in Computer Science, Vol. 7237 (2012) 719 -- 737.

\bibitem{BV}
Z. Brakerski and V. Vaikuntanathan, \emph {Fully homomorphic encryption from Ring-LWE and security for key dependent messages}, {\bf Lecture Notes in Computer Science} Vol. 6841, (2011), 505-524.

\bibitem{provably weak revisited}
W. Castryck, I. Iliashenko, and F. Vercauteren, \emph{Provably Weak Instances of Ring-LWE Revisited}, Advances in Cryptology - CRYPTO 2016, \textbf{Lecture Notes in Computer Science}, Vol. 9665, Springer (2016), 147-167.

\bibitem{attacks on search}
H. Chen, K. Lauter, and K.E. Stange, \emph{Attacks on the Search-RLWE problem with small errors}, SIAM Journal on Applied Algebra and Geometry, Vol. 1 (2017).

\bibitem {CLS} H. Chen, K. Lauter, and K. E. Stange, \emph{Attacks on the Search RLWE Problem with Small Errors}, \textbf{SIAM J. Appl. Algebra Geometry}, Vol. 1(1), 665--682. 

\bibitem {DPA} 
I. Damg\"{a}rd, A. Polychroniadou, and R. Adaptively, \emph {Secure Multi-Party Computation from LWE}, {\bf PKC 2016: Public-Key Cryptography},  Lecture Notes in Computer Science, Vol. 9615,  208--233. 


\bibitem{RLWE for NT}
Y. Elias, K.E. Lauter, E. Ozman, and K.E. Stange, \emph{Ring-LWE Cryptography for the Number Theorist}, Directions in Number Theory, \textbf{Association for Women in Mathematics Series}, Vol. 3, Springer (2016), 271--290.

\bibitem{provably weak}
Y. Elias, K.E. Lauter, E.Ozman, and K.E. Stange, \emph{Provably Weak Instances of Ring-LWE}, Advances in Cryptology – CRYPTO 2015, \textbf{Lecture Notes in Computer Science}, Vol. 9215,  Springer, Heidelberg (2015), 63--92.

Y. Elias, K.E. Lauter, E. Ozman, and K.E. Stange, \emph{Ring-LWE Cryptography for the Number Theorist}, Directions in Number Theory, \textbf{Association for Women in Mathematics Series}, Vol. 3, Springer (2016), 271--290.


\bibitem{BSNK}
K. Basu, D.  Soni, M. Nabeel, and R.  Karri, \emph{NIST Post-Quantum Cryptography: A Hardware Evaluation Study}, {\bf IACR Cryptology ePrint Archive}, Vol. 47 (2019).



\bibitem {HPS} J. Hoffstein, J. Pipher, J. H. Silverman, \emph{NTRU: A Ring Based Public Key Cryptosystem},  {\bf Lecture Notes in Computer Science} Vol. 1423,  (1998), 267--288.


\bibitem{a lower bound}
Y. P. Hong and C.-T.Pan, \emph{A Lower Bound for the Smallest Singular Value}, Linear Algebra and its Applications, Vol. 172 (1992), 27--32.



\bibitem{LP}
R. Lindner and Chris Peikert, \emph{ Better Key Sizes (and Attacks) for LWE-Based Encryption}, {\bf Lecture Notes in Computer Science}, Vol.6558, (2011), 319--339. 

\bibitem{on ideal lattices}
V. Lyubashevsky, C. Peikert, and O. Regev, \emph{On Ideal Lattices and Learning with Errors Over Rings}, Advances in Cryptology – EUROCRYPT 2010: 29th Annual International Conference on the Theory and Applications of Cryptographic Techniques, French Riviera, May 30 – June 3, 2010. Proceedings (2010), 1-23.


\bibitem{toeplitz}
M.W. Meckes, \emph{On the Spectral Norm of a Random Toeplitz Matrix}, {\bf Electronic Communications in Probability},  Vol. 12 (2007),  315--325. 

\bibitem{MR}
D. Micciancio and O. Regev, \emph{Lattice-based cryptography}, {\bf  Advances in Cryptology - CRYPTO 2006}, (2009), 131--141.


bibitem{LP}
O. Regev, \emph{On Lattices, Learning with Errors, Random Linear Codes, and Cryptography}, {\bf Journal of the ACM (JACM)}, Vol. 56: 6, (2009),  84--93.


\bibitem{note}
Y. Yu and D. Gu, \emph{A note on a lower bound for the smallest singular value}, Linear Algebra and its Applications, Vol. 252
(1997), 25–38.

\bibitem {W} T. Wang, J. Yu, P.  Zhang and Y. Zhang, \emph {Efficient Signature Schemes from R-LWE}, {\bf Trans. Internet Inf. Syst.}, Vol. 10 (2010), 3911--3924.

\end{thebibliography}
\end{document}